\documentclass[oneside]{amsart}

\usepackage{subfiles}	
\usepackage[utf8]{inputenc}
\usepackage[english]{babel}
\usepackage{amsmath,amsthm}
\usepackage{amsfonts}
\usepackage{amssymb}
\usepackage{graphicx}
\usepackage[]{geometry}
\usepackage{tikz}
\usepackage{tikz-3dplot,pgfplots}
\usepackage{color}
\usepackage{enumerate}
\usepackage[all]{xy}
\usepackage[normalem]{ulem}
\usepackage{cleveref}
\usepackage{overpic}

 \newtheorem{theorem}{Theorem}[section]

  \newtheorem{corollary}[theorem]{Corollary}
  \newtheorem{question}[theorem]{Question}
 \newtheorem{lemma}[theorem]{Lemma}
  \newtheorem*{claim}{Claim}
  \newtheorem{proposition}[theorem]{Proposition}

   \theoremstyle{definition}
\newtheorem{definition}[theorem]{Definition}

\theoremstyle{remark}

\newcommand{\HH}{\ensuremath{{\mathbb{H}}}}
\newcommand{\Z}{\ensuremath{{\mathbb{Z}}}}
\newcommand{\R}{\ensuremath{{\mathbb{R}}}}

\newcommand{\Q}{\ensuremath{{\mathbb{Q}}}}

\newcommand{\InfSum}{\ensuremath{ \bigoplus \Z_n}} 
\newcommand{\InfSumtwo}{\ensuremath{ \bigoplus \Z_2}}

 \usepackage{enumitem}
\usepackage{xcolor}

\begin{document}

%\title{A note on pattern preserving quasi-isometries in lamplighter groups, Baumslaug Solitar Groups and SOL}
\title[Pattern preserving quasi-isometries in lamplighter groups]{Pattern preserving quasi-isometries in lamplighter groups and other related groups}
%\author{Tullia Dymarz, Beibei Liu, Nata\v{s}a Macura, Rose Morris-Wright}

\author{Tullia Dymarz}
\address{ Tullia Dymarz,
Department of Mathematics,
University of Wisconsin-Madison, 480 Lincoln Drive,  Madison, WI 53706}
\email{dymarz@math.wisc.edu}

\author{Beibei Liu}
\address{Beibei Liu, Department of Mathematics, Ohio State University, 100 Math Tower, 231 W 18th Ave, Columbus, OH 43210}
\email{liu.11302@osu.edu}
\author{Nata\v{s}a Macura}
\address{Nata\v{s}a Macura (Emerita) Trinity University 
1 Trinity Place
San Antonio TX 78212}
\email{natasa.nmacura@gmail.com}
\author{Rose Morris-Wright}
\address{Rose Morris-Wright, Department of Mathematics and Statistics, Middlebury College, Warner Hall, Middlebury, VT 05753} \email{rmorriswright@middlebury.edu} 

\begin{abstract}
%Lamplighter groups, Baumslag-Solitar groups and SOL are all examples of solvable groups whose geometries can be thought of as subsets of a product of $\delta$-hyperbolic spaces. Such spaces admit many quasi-isometries and so we will restrict to quasi-isometries of these spaces that preserve a particular set of cosets in the group. We give a comparative study of how such pattern preserving quasi-isometires behave in all three groups, including an expository review of the results in the Baumslaug Solitar and SOL cases, and a discussion of new results which show how Lamplighters groups behave similarly, but are also much less rigid.  
%\commentTullia{Rewrite}
In this paper we explore the interplay between aspects of the geometry and algebra of three families of groups of the form $B\rtimes \Z$, namely Lamplighter groups, solvable Baumslag-Solitar groups and lattices in SOL. In particular we examine what kind of maps are induced on $B$ by quasi-isometries that coarsely permute cosets of the $\Z$ subgroup.  By the results of Schwartz(1996) and Taback(2000) in the lattice in SOL and solvable Baumslag-Solitar cases respectively such quasi-isometries induce affine maps of $B$. We show that this is no longer true in the lamplighter case but the induced maps do share some features with affine maps.

\end{abstract}
\maketitle
\section{Introduction} \label{Section: Introduction}

A quasi-isometry is a map between metric spaces that preserves the large scale geometry, only allowing for bounded discontinuities or for distances to be scaled by a bounded amount. In geometric group theory, we aim to study groups via this large scale geometry. We consider a finitely generated group as a metric space by endowing each group element with a length corresponding to the minimal number of generators in a word representing the given group element. Two different presentations for the same group induce different word metrics but all such metrics are quasi-isometric. This means that differences in large scale geometry correspond to non-isomorphic groups. However quasi-isomorphic metric spaces do not always come from isomorphic groups. Gromov's program aims to classify all finitely generated groups up to quasi-isometry (See \Cref{Section: Background} for technical details on quasi-isometries and word metrics.)

 Often a first step  in understanding which groups are quasi-isometric to a given group $G$ is to study the structure of self quasi-isometries of $G$. A priori, self quasi-isometries may seem unwieldy so to get a handle on them one looks for structures that they preserve. For example,
 given a subgroup $H$, we say that a self quasi-isometry $\Psi:G\to G$ \emph{coarsely permutes} the cosets of $H$ if there exists a permutation $\sigma:G/H\to G/H$ and a $C'\geq 0$ such that $d_{\mathcal{H}}(\Psi(gH), \sigma(gH))<C'$ where $d_{\mathcal{H}}$ denotes Hausdorff distance. We call such quasi-isometries \emph{pattern preserving}.
 Even if not all self quasi-isometries  coarsely permute the cosets of a given subgroup (i.e. are pattern preserving), one can try to understand how much smaller is the set of self quasi-isometries that do. Answering these types of questions can often lead to rigidity results for groups built out of $G$ through graphs of group constructions (see \cite{Cashen}) and has also become a question of independent interest.  For a nice survey on some results and questions in this direction see \cite{Eduardo}.

In this paper we focus on three classes of groups of the form $G=B \rtimes \Z$ that have similar geometry: 
\begin{enumerate}

\item {\bf Lamplighter group.} $L_n:= \InfSum\rtimes \Z$ where semi-direct product is given by the action shifting the index.
\item {\bf Lattices in SOL.} $\Gamma_A: = \Z^2 \rtimes_A \Z$ where the action is by multiplication by a hyperbolic matrix  $A \in SL(2,\Z)$.
\item {\bf Solvable Baumslag-Solitar group.}  $BS(1,n) := \Z[\frac{1}{n}] \rtimes \Z$ where the action is by multiplication by $n$.
\end{enumerate}
In each case we call $B$ the \emph{base group} and the subgroup generated by $\Z$ the \emph{vertical subgroup}.
By theorems of Farb-Mosher \cite{Farb1998,Farb1999} in the solvable Baumslag-Solitar group case and Eskin-Fisher-Whyte \cite{Eskin2012,Eskin2013} in the other two cases, all quasi-isometries coarsely permute left cosets of $B$. 

In this paper we are interested in quasi-isometries  that coarsely permute the left cosets of the vertical subgroup $\Z$. Note that the space of cosets of the vertical subgroup can be identified with the base group $B$ and by assumption any pattern preserving quasi-isometry $\Psi$ induces a permutation $\psi:B \to B$. 
 A priori, this induced map does not need to have any structure. However by work of Schwartz \cite{Schwartz1996} in the case of $\Gamma_A$ and Taback \cite{Taback2000} in the case of $BS(1,n)$ it turns out that the induced maps on $B=\Z^2$ and $B=\Z[\frac{1}{n}]$ respectively are algebraic in nature. In particular they satisfy the property that for any $a,b,c,d\in B$ with $a+c=b+d$  the image under the induced map $\psi$ satisfies
\[\psi(a)+\psi(c)=\psi(b)+\psi(d)\] where `+' represents the operation in the group $B$. If we imagine the four group elements $a,b,c,d$ as the sides of a quadrilateral then we can view the condition that $a+c=b+d$ as implying that the quadrilateral is a parallelogram. We therefore call the above condition on $\psi$ \emph{parallelogram preserving} since it states that $\psi$ takes parallelograms in $B$ to parallelograms.

In the case of $BS(1,n)$ and $\Gamma_A$, the map $\psi$ being parallelogram preserving is equivalent to $\psi$ being affine on $\Z[1/n]$ or $\Z^2$ respectively.
If we consider parallelogram preserving maps on B in  the lamplighter case where $B= \InfSum$ we already encounter different behavior from the other two groups. For both $BS(1,n)$ and $\Gamma_A$ the action of the vertical subgroup $\Z$ on their respective base groups $B$ is affine ($x \to nx$ in the $BS(1,n)$ case and $x \to Ax$ in the $\Gamma_A$ case) while in the lamplighter group the shift map is not. We are interested in studying a set of maps that includes all isometries induced by the group action, so in order to allow for left multiplication by group elements in all three cases, we call a map of $\InfSum$ \textit{generalized affine} if it is the composition of an affine map and a vertical shift map $(x_i)\mapsto (x_{i+k})$ for some constant $k\in \Z$. With this definition, left multiplication by a group element in $L_n$ will induce a generalized affine map on $\InfSum$. 

Our main question is whether or not all pattern preserving quasi-isometries are generalized affine. This is true for lattices in SOL and for Baumslag-Solitar groups by  \cite{Schwartz1996,Taback2000} but the situation is more complicated for Lamplighter groups. This paper looks at the proofs of these known results from a geometric standpoint and discusses how the shared geometry of these three groups might lead one to expect that the lamplighter group has similar behavior. It turns out, however, that pattern preserving quasi-isometries behave  differently in lamplighter groups and this paper describes their behavior and explains why it is different from the first two cases. 

 In order to have any hope of finding conditions in that lamplighter group $L_n$ that will guarantee generalized affine behavior,  we must first restrict our attention to $n=2$.
For $L_n$ with $n \geq 3$ there are examples of pattern preserving quasi-isometries (and even isometries if we chose our generating set correctly) that do not induce generalized affine maps on $B= \InfSum$.
For example at any index simply take a permutation of $\Z_n$ that does not come from an affine map of $\Z_n$.  The situation for $n=2$ is more interesting.  %For $L_2$ we have the following 
First we examine pattern preserving isometries. Since different word metrics are not isometrically equivalent we fix a generating set for $L_2$ whose Cayley graph is the Diestel-Leader graph $DL(2,2)$ (see Section \ref{Section: Background} for details).

\begin{theorem}\label{thrm:affineisometry} If $\Psi:L_2\to L_2$ is a pattern preserving isometry then the induced map $\psi$ comes from left multiplication by a group element (up to composition with the map induced by inversion $(x_i)_i \mapsto (x_{-i})_i$), and hence is a generalized affine map.
\end{theorem}
When we generalize to arbitrary quasi-isometries on $L_2$, we show that the induced maps are much less rigid than in $BS(1,n)$ or  $\Gamma_A$.
 
\begin{theorem} \label{thrm:counterexamples} There are pattern preserving quasi-isometries of $L_2$ that do not induce generalized affine maps of $\InfSumtwo$.
\end{theorem}

In \Cref{subsec:counterexample}, we explicitly construct a family of maps $\psi: \InfSumtwo\to \InfSumtwo$  that are not affine, but which induce pattern preserving quasi-isometries on $L_2$.
Despite the less rigid structure implied by \Cref{thrm:counterexamples}, we are still able to show that quasi-isometries that preserve cosets of the vertical subgroup are ``affine on the large scale." That is they are affine when only applied to ``large" enough parallelograms. These notions are developed in Section \ref{subsec:eMquads} and the result is proved as Proposition \ref{Prop:eMlamplighter}
but the key ingredient, captured in the Lemma \ref{Theorem: Main} below, boils down to the fact that all large enough quadrilaterals must be parallelograms.

\begin{lemma}\label{Theorem: Main} 
Let $S>1$ and for $(x_i) \in \InfSumtwo$ denote $supp((x_i))$ to be the smallest interval containing all indices $i$ for which $x_i \neq 0$. Suppose that $a,b,c,d \in \InfSumtwo$ satisfy the following

\begin{enumerate}
\item $|supp(b-a)|<S, |supp(c-a)|<S$
\item $ |supp(b-c)|>2S, |supp(d-a)|>2S$
\end{enumerate}
Then $a+b=c+d$.
\end{lemma}

This result is analogous to the first step of the proofs in \cite{Schwartz1996,Taback2000} in the $\Gamma_A$ and $BS(1,n)$  cases and follows from the common geometry of all three groups. 
We will describe the common geometry  shared by all three groups in \Cref{Section: Background}. In \Cref{Section: BSandSOL}, we will compare the proofs of the results from Taback for $BS(1,n)$ \cite{Taback2000}, and Schwartz for lattices in SOL \cite{Schwartz1996}, with a focus on how these proofs reflect this shared geometry. Finally, in \Cref{Section: MainResults} we pinpoint where the arguments of Taback and Schwartz break down in the lamplighter case and prove the main results, Theorem \ref{thrm:affineisometry}, \ref{thrm:counterexamples}, and   \ref{Theorem: Main}.

\vspace{5mm}

 \textbf{Acknowledgment}: The first author was partially supported by NSF CAREER 1552234. The second author is partially supported by the NSF grant DMS-2203237. The authors would like to thank the organizers of Women in Groups, Geometry and Dynamics funded by NSF CAREER 1651963 for giving them the opportunity to come together to work on this project. The first author would also like to thank Jennifer Taback and Kevin Whyte for many useful conversations.

\section{Background} \label{Section: Background}

\subsection{Lamplighter groups, lattices in SOL, and Baumslag-Solitar groups}
 We are interested in three different types of groups, all of which are commonly studied in geometric group theory. Below we establish notation and define each group as a semi-direct product of some base group $B$ with $\Z$.

 {\color{red}}
 \subsubsection{Lamplighter groups} 
 Let $\Z_n$ be a finite cyclic group with $n$ elements. You can apply this construction for any finite group, but for simpilicity we will focus on $\Z_n$. Suppose that we have a bi-infinite sequence of lamps, one for each integer on the real line. Each lamp has $n$ settings corresponding to the $n$ elements in $\Z_n$. Each element $g$ of the lamplighter group, $L_n$, is written $g=((x_i),k)$, where $(x_i)\in\InfSum$ represents the locations of finitely many lamps on non-zero settings, and $k\in \Z$ represents the location of the lamplighter. 

To compose two elements, we think of the element $g=((x_i),k)$ as a set of instructions, telling the lamplighter to start at position 0, go flip the switch to setting $x_i$ at each position $i$, where $x_i\neq 0$ and then end at position $k$. Composition in the group can be thought of as concatenating two sets of such instructions, such that lamplighter begins the second set of instructions at position $k$. So if $g=((x_i),k)$ and $h=((y_i),\ell)$ then 
\[gh=((x_i+y_{i-k}),k+\ell)\]

The group $L_n$ is then a semi-direct product $\InfSum \rtimes \Z$. We will use a presentation of this group with generators $t,at$ where $t$ represents the instruction for the lamplighter to move one step in the positive direction without changing the lamp, each $x$ corresponds to a generator of $\Z_n$, so $at$ represents the instructions for the lamplighter to change the setting of the current lamp accordingly and then move one step in the positive direction. This gives the following presentation. 

\[L_n=\langle t,at\mid a^n=1, [t^ia t^{-i}, t^ja t^{-j}] \quad \forall i,j\in \Z\rangle\]

Notice that this presentation has infinitely many relations, and in fact the lamplighter group is not a finitely presented group. The Cayley graph given by this presentation is the Diestel-Leader graph discussed below. 

\subsubsection{Lattices in SOL}
Consider the semi-direct product $ \R^2 \rtimes \R$  where $\R$ acts on $\R^2$ by $(x,y) \to (e^{t}x, e^{-t}y)$. This is the three-dimensional solvable Lie group $SOL$.
Let $A\in SL(2, \Z)$ be  diagonalizable with no eigenvalues on the unit circle, and  $\Gamma_{A}=\mathbb{Z}^{2}\rtimes \Z$ where $\Z$ acts on $\Z^{2}$ by powers of $A$. We can see that $\Gamma_{A}$ is a lattice in $SOL$ with the embedding  given by $(x,y,z) \mapsto (Q(x,y),z)$ where $QAQ^{-1}$ is a diagonal matrix. The quotient of $SOL/\Gamma_A$ is compact so that $\Gamma_A$ is a uniform lattice in $SOL$.

\subsubsection{Solvable Baumslag-Solitar groups}

Solvable Baumslag-Solitar groups are groups given by the following presentation.

\[BS(1,n)=\langle t,a \mid tat^{-1}=a^n\rangle.\]

Alternatively, we can also view $B(1, n)$ as the group of affine maps on $\Z[\frac{1}{n}]$, with $a$ acting as $i\rightarrow i+1$ and $t$ as $i\rightarrow \frac{1}{n}i$, which determines the short exact sequence 
$$1\rightarrow \Z[\tfrac{1}{n}]\rightarrow BS(1, n)\rightarrow \mathbb{Z}\rightarrow 1.$$
Thus, we can also identify $B(1, n)=\mathbb{Z}[\frac{1}{n}]\rtimes \Z $.

\subsection{Quasi-isometries}
\begin{definition}[Quasi-isometry] For $K\geq 1, C\geq 0$ and $(X,d_X), (Y,d_y)$ metric spaces we say that $f:(X,d_X) \to (Y,d_Y)$ is a $(K,C)$ \emph{quasi-isometry} if for all $x,x' \in X$
\begin{enumerate}
\item $-C + \frac{1}{K} d_X(x,x') \leq d_Y(f(x),f(x')) \leq K d_X(x,x') + C$
\item for all $y \in Y$ there is $x \in X$ such that $d_Y(f(x),y ) \leq C$.
\end{enumerate}
\end{definition} 

If only the first condition is satisfied we say that $f$ is a $(K,C)$ \emph{quasi-isometric embedding} and if only the second condition is satisfied we say that $f$ is \emph{coarsely surjective}. 
We say that two spaces are \emph{quasi-isometrically equivalent} if for some $K\geq 1$ and $C\geq 0$ there is a $(K,C)$ quasi-isometry between them and in most cases we don't specify the constants $K$ and $C$. A \emph{coarse inverse} of a quasi-isometry $f:X \to Y$ is a quasi-isometry $\bar{f}:Y \to X$ such that both  $d_{sup}(f\circ \bar{f} , Id_Y) < \infty$ and $d_{sup}(\bar{f}\circ f , Id_X) < \infty$ where $d_{sup}$ is the sup metric and $Id_X, Id_Y$ are the identify maps on $X$ and $Y$ respectively. More generally we can define an equivalence class on self-quasi-isometries of a metric space $X$ where for $f,g: X \to X$ we set $f\sim g$ if $d_{sup}(f,g)< \infty$. The set of these equivalence classes is a group under the composition of representatives which we call 
 the \emph{quasi-isometry group} $QI(X)$.
 
 In the special case when $C=0$ then we say that $f$ is a \emph{biLipschitz} map and if in addition $K=1$ then $f$ is an \emph{isometry}. We can also define the group of self-biLipschitz maps of a metric space $Bilip(X)$ or self-isometries $Isom(X)$ but for these we do not need to look at equivalence classes and in general these groups do not embed into $QI(X)$. 
 
 \subsubsection{Finitely generated groups as metric spaces.} A finitely generated group can be treated as a metric space by fixing a finite generating set  $S$ and declaring $d_S(g,h)=\| g^{-1} h \|_S$ where $\| x \|_S$ counts the minimum number of generators of $S$ or their inverses  needed to write the word $x$. This distance depends on $S$ but all such metrics are \emph{quasi-isometric} in that  if $S$ and $S'$ are both finite generating sets for $G$ then the identity map $Id: (G,d_S) \to (G,d_{S'})$ is a quasi-isometry.

\subsubsection{Pattern preserving quasi-isometries}
As outlined in the introduction if a quasi-isometry $\Psi:G \to G$ coarsely permutes left cosets of a certain subgroup $H \leq G$ then we can call it pattern preserving. In this paper we focus on the case of $G=B \rtimes \Z$ and $H=\Z$.

   \begin{definition} If $G=B\rtimes \Z$ is a semi-direct product, we say that a quasi-isometry $\Psi: G \to G$ is a \textit{pattern preserving quasi-isometry} if it coarsely permutes the left cosets of $H=\Z$.
      \end{definition}

\subsection{Model spaces and Horocyclic products.}
 
 \begin{definition} We say that a proper geodesic metric space $X$ is a \emph{model space} for a finitely generated group $G$ if $G$ acts on $X$ properly discontinuously and cocompactly by isometries. 
 \end{definition}
 
By the Fundamental Lemma of geometric group theory if $X$ is a model space for $G$ then $G$ with any word metric is quasi-isometric to $X$. 
 Any Cayley graph of $G$ with respect to a finite generating set is a model space for $G$ and if $G$ is a cocompact lattice in a Lie group then this Lie group with any left invariant Riemannian metric is also a model space for $G$.
 
 For the three types of groups we are interested in (lamplighter groups, solvable Baumslag-Solitar groups, and $\Gamma_A$ lattices in SOL), we can construct model spaces as \textit{horocylic products}. 
 
 \begin{definition}Given two metric spaces $X_1,X_2$, each equipped with a surjective continuous function $h_i:X_i\to \R$, called the \textit{height functions}, the \textit{horocyclic product} is a subset of the direct product $X_1 \times X_2$ given by 
 \[X_1\times_h X_2:= \{(x,y)\in X_1\times X_2|h_1(x)+h_2(y)=0\}.\]
 \end{definition}

 In this paper we will  consider at most the case when $X_1,X_2$ are $CAT(-1)$ spaces and $h_i$ are Busemann functions and most often the case when the $X_i$ are either regular trees or a hyperbolic plane.
 We often draw the two spaces $X_1$ and $X_2$ so that the height function for $X_1$ corresponds to moving upward on the page and the height function for $X_2$ corresponds to moving downward on the page. Using this visualization, points in the horocylic product correspond to pairs of points at the same height on the page (see \Cref{fig_horocylicProduct}. We can then define a height function $h: X_1\times_h X_2\to \R$ given by $h_1=-h_2$. 

\begin{figure}
    \centering
    \includegraphics[scale=0.16]{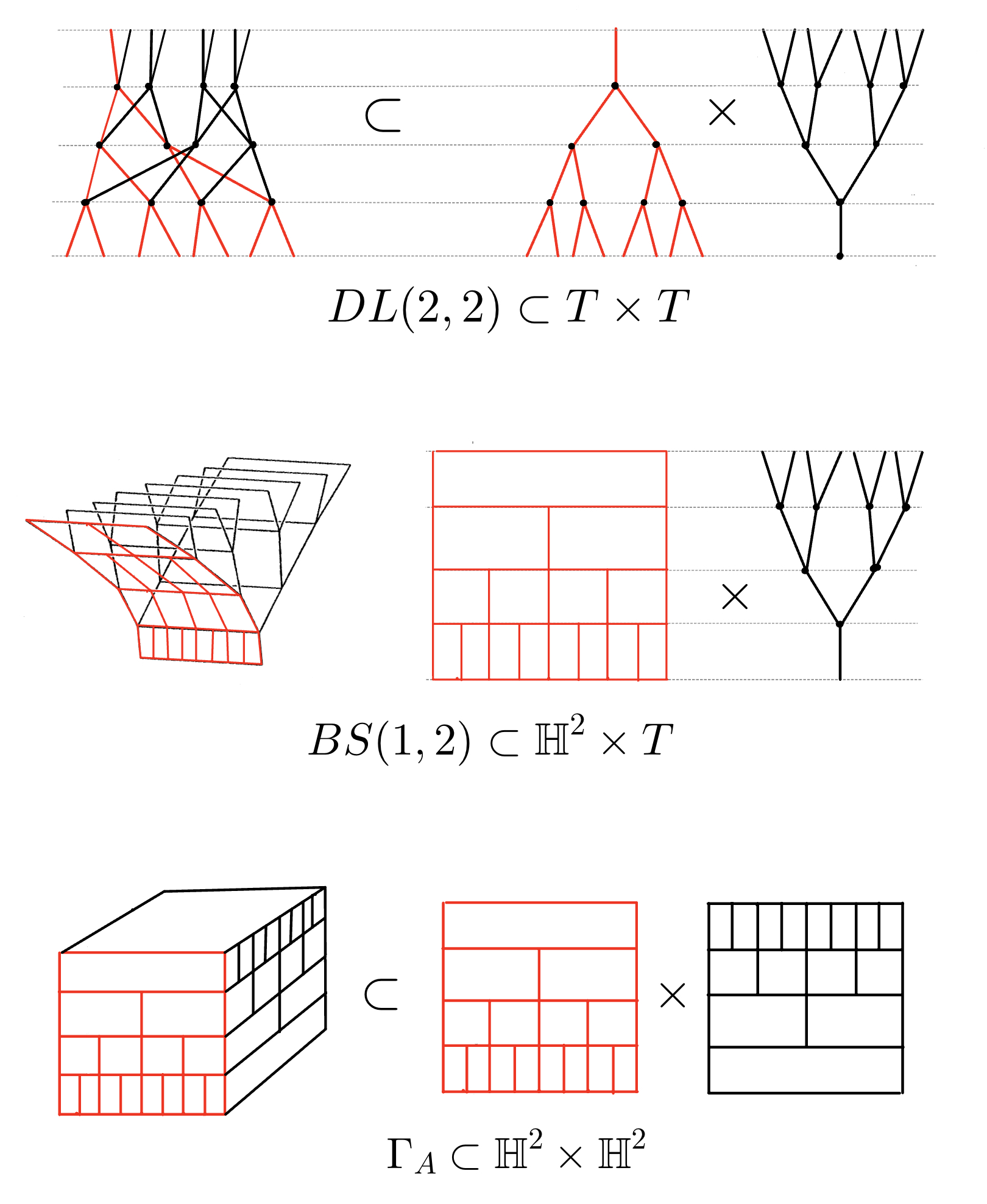}
    \caption{\emph{Top:} The Diestel-Leader graph as a horocyclic product two trees.\\
    \emph{Middle:} The model space for $BS(1,2)$ as a horocyclic product of the hyperbolic plane with a tree. \\
    \emph{Bottom:} The model space of a lattice in SOL as a horocyclic product of two hyperbolic planes. Note a part of this figure is taken from a figure in \cite{Farb1998}
    %\commentRose{I realized that it should maybe say $T_2$ instead of $T^2$? Or do we want $T_3$? BL: I also have a small question for the figure.}
    %\commentTullia{We should write $T$?}
    }
    \label{fig_horocylicProduct}
\end{figure}

  We let $T_{n+1}$  be the $n+1$ valent oriented tree with $n$ incoming edges and one outgoing edge. This orientation allows us to define a height by fixing a basepoint $v_0 \in T_{n+1}$ and letting $h: T_{n+1} \to \R$ be the unique map with $h(v_0)=0$ that maps vertices to integers, edges to intervals and preserves the orientation in that if $(v,w)$ is an edge oriented from $v$ to $w$ then $h(w)=h(v)+1$. 
  
 We use the log model of the hyperbolic plane  $\mathbb{H}^2$ so that if $(x,y)$ represent the usual upper half plane coordinates with metric $ds^2 = (dx^2 + dy^2)/y^2$ we use instead $(x,t)$ where $t=\ln{y}$ in which case the metric is given by $ds^2= e^{-2t}dx^2 +dt^2$. This allows us to define the \emph{height} of point $h: \mathbb{H}^2 \to \R$ as $h(x,t)=t$.

 With this notation we can create model spaces for our three families of groups as follows. 
 
 \begin{itemize}
 \item \textbf{lamplighter groups $L_n$:} The Cayley graph given by the above presentation, also called the Diestel-Leader graph, is the horocyclic product \[DL(n,n)=T_{n+1}\times_hT_{n+1}.\] In this graph, two vertices $(x_1,y_1),\ (x_2,y_2)\in DL(n,n)$ are joined by an edge whenever both  $x_1$, $x_2$ and $y_1$, $y_2$ are joined by an edge in their respective trees. 
  \item \textbf{lattices in SOL, $\Gamma_A$:} A model space can be constructed as a lattice inside the horocyclic product 
 \[\HH^{2}\times_h \HH^{2}.\]
 This space is isometric to $SOL= \R^2 \rtimes \R$ with the standard Riemannian metric 
 $$ds^2= e^{-2t}dx^2 + e^{2t}dy^2 +dt^2.$$
 \item \textbf{solvable Baumslag-Solitar groups, BS($1,n$):} The model space constructed by Farb and Mosher in \cite{Farb1998} is a horocylic product. \[X_n=\HH^{2}\times_h T_{n+1}.\]
The one caveat here is that we have to either rescale the edge lengths in the tree to be $\ln{n}$ or rescale the curvature of the hyperbolic plane. See \cite{Farb1998}. 
 \end{itemize}

\subsubsection{Distance in horocyclic products.} In the Diestel-Leader graph $DL(n,n)$, distance is given by 
$$d((x_1, x_2), (y_1, y_2))= d_{T_{n+1}}(x_1,y_1)+d_{T_{n+1}}(x_2, y_2) - |h(x_1)-h(y_1)|$$
where $d_{T_{n+1}}$ is distance in the tree $T_{n+1}$ and $h$ is the height function on the horocyclic product \cite{Bertacchi}. On more general horocyclic products a similar formula coarsely holds, see Theorem A in  \cite{Ferragut}. %See also \cite{Riley} for more details about the geometry of lamplighter groups.
In the case of $SOL$ any two points at the same height $t$ can be connected by a path that lies at height $t$ and so we can compute induced distances along height level sets. These are given by $d_t((p,t),(q,t))=e^t|\Delta x| + e^{-t}|\Delta y|$ see for example \cite{Eskin2012}.

 \subsection{Vertical geodesics and boundaries}
 \begin{definition}
Given a space $X$ with a height function $h:X \to \R$ we can define a \emph{vertical geodesic}
as a geodesic in $X$ that maps isometrically onto $\R$ via the height map $h$.     
 \end{definition}

Let $\mathcal{L}$ denote the space of all {vertical geodesics} in the space $X$.
Given $\gamma,\gamma'\in \mathcal{L}$, there are two natural equivalence relations $\sim_u$ and $\sim_\ell$ on $\mathcal{L}$:
$$\gamma \sim_u \gamma' \text{ if and only if } \lim_{t\to \infty} d(\gamma(t), \gamma'(t)) < \infty$$
and
$$\gamma \sim_\ell \gamma' \text{ if and only if } \lim_{t \to -\infty} d(\gamma(t), \gamma'(t)) < \infty.$$
We then define the \emph{lower boundary} of $X$ to be
$$\partial_\ell X := \mathcal{L}/\sim_\ell$$ and the \emph{upper boundary} to be 
$$
\partial^u X := \mathcal{L}/\sim_u.$$

For example, on $\mathbb{H}^2$ with its standard height function we have $\partial_\ell \mathbb{H}^2 \simeq \R$ while $\partial^u \mathbb{H}^2$ is a single point $\{ \infty \}$. 
In general, for a Gromov hyperbolic space $X$ with height $h$ given by a Busemann function, you can identify $\partial_\ell X= \partial_\infty X \setminus \{\infty \}$ where $\partial_\infty X$ is the usual visual boundary of $X$ and $\partial^u X= \{\infty\}$.

In the case of $T_{n+1}$ we have a nice identification of $\partial_\ell T_{n+1}$ with power series $\sum_{i=L}^\infty a_i n^i$ where $L \in \Z$  and $a_i \in \Z_n$, that is, with elements of $\Q_n$. To see this identification, first for any vertex $v\in T_{n+1}$ label all $n$ edges coming to $v$ with labels $0,1, \ldots, n-1$. Then a vertical geodesic determines a bi-infinite sequence $\{a_i\}$. Since any vertical geodesic must eventually join the geodesic labeled with all $0$'s we can identify this sequence with $\sum_{i=L}^\infty a_i n^i$ for some $L \in \Z$.

\begin{lemma} If $X_1\times_h X_2$ is a horocyclic product of Gromov hyperbolic spaces $X_1, X_2$ with height functions given by Busemann functions then we have the following boundary identifications
$$\partial_\ell (X_1\times_h X_2) =\partial_\ell X_1
\textrm{ and }\partial^u (X_1\times_h X_2) =\partial_\ell X_2.$$    
\end{lemma}
\begin{proof}
 This follows easily from the fact that the height function in the second factor is reversed. 
\end{proof}

\begin{corollary} The lower and upper boundaries of $L_n, BS(1,n)$ and  $ \Gamma_A$ are as follows:
    
 \begin{itemize}
 \item \textbf{For lamplighter groups $L_n$:} 
$$\partial_\ell (DL(n,n)) =\partial_\ell T_{n+1}\simeq \Q_n, \ \partial^u (DL(n,n)) =\partial_\ell T_{n+1}\simeq \Q_n$$

 \item \textbf{For solvable Baumslag-Solitar groups $BS(1,n)$:} 
 $$\partial_\ell (X_n) =\partial_\ell \mathbb{H}^{2}\simeq \R, \ \partial^u (X_n) =\partial_\ell T_{n+1}\simeq \Q_n$$

 \item \textbf{For $\Gamma_A$ lattices in $SOL$:} 
  $$\partial_\ell (SOL) =\partial_\ell \mathbb{H}^{2}\simeq \R,\ \partial^u (SOL) =\partial_\ell \mathbb{H}^{2}\simeq \R$$
 
 \end{itemize}
 \end{corollary}

We also need to endow these boundaries with a metric. When $X$ is $CAT(-1)$, such as in our case, there is a well studied family of metrics on $\partial_\ell X$: the \emph{parabolic visual boundary metrics}. 
In the case of $\mathbb{H}^2$ and $T_{n+1}$, the usual metrics on $\R$ and $\Q_n$, defined by the Euclidean and $n$-adic norms respectively, are examples of parabolic visual metrics so we will just use those. Geometrically we should think of these metrics as $d(x,x')=a^{t_0}$ where $t_0$ is the height at which the vertical geodesics defined by $x,x'$ come close and $a>0$ is a parameter which in the case of $T_{n+1}$ can be chosen to be $n$ and in the case of the curvature $-1$ hyperbolic plane $a=e$. The notion of coming close is also a little different in the two settings. In the $T_{n+1}$ case coming close will mean intersecting for the first time while in the hyperbolic plane we mean coming within some fixed bounded distance of each other. This last notion also works for general $CAT(-1)$ spaces so anytime we use these boundary metrics we assume that the $X_i$ are $CAT(-1)$ spaces.

\subsection{$(\epsilon,M)$-quadrilaterals}\label{subsec:eMquads} 

Our next step is to understand how vertical geodesics in our horocyclic product come together.
We will be able to do this in terms of the boundary metrics which we denote 
$d_\ell$ and $d_u$ for the lower and upper boundaries respectively.
%or this we introduce the following definition. 

\begin{definition}
    Given vertical geodesics $p,q$ in a horocyclic product $X_1\times_h X_2$, we define $$\delta(p,q):=d_\ell(p,q)d_u(p,q).$$ 
\end{definition}

In the lamplighter case $\log_n(\delta(p,q))$ has a very nice geometric interpretation, specifically it measures the distance between the vertical geodesics $p$ and $q$ inside the Diestel-Leader graph (see Figure \ref{fig_vertgeods} for an explanation).

\begin{figure}\label{fig_vertgeods}
\begin{center}
\includegraphics[width=1.8in]{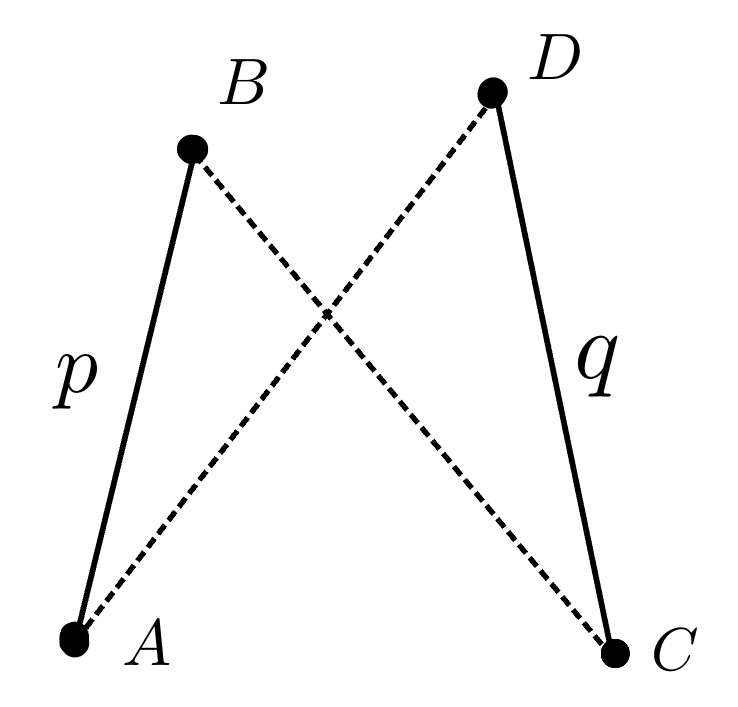}
\caption{The distance between geodesics $p$ and $q$ is given by the length of the segment from $B$ to $C$ (or equivalently $A$ to $D$). If $h$ denotes the height function on $DL(n,n)$. Then $d_{\ell}(p,q)=n^{h(B)}$ and $d_u(p, q)=n^{- h(C)}$ so that $\delta(p,q)=n^{h(B)-h(C)}$ and therefore $\log_n(\delta(p,q))$ is exactly the distance between the geodesics $p$ and $q$. 
}
\label{fig_vertgeods}
\end{center}
\end{figure}

For more general horocyclic products a coarse version of this interpretation can be deduced 
using Theorem B from \cite{Ferragut}. In particular the picture in Figure \ref{fig_vertgeods}  needs to be modified slightly so that the points $A,B,C,D$ are now replaced by balls of a fixed radius. 
Theorem B ensures that the shortest path between the vertical geodesics $p$ and $q$ stays bounded distance from vertical geodesics that either start in $A$ and end in $D$ or start in $B$ and end in $C$. 

The height interval at which this geodesic between $p$ and $q$ lies is also important. It coincides coarsely with the interval defined by $\log{d_\ell(p,q)}$ and $-\log{d_u(p,q)}$ where the logarithm base is determined by the base chosen for the boundary metrics. 
\begin{definition}\label{defn:coarseheight} Suppose $X=X_1 \times_h X_2$ is a horocyclic product of two $CAT(-1)$ spaces $X_i$ and $d_\ell, d_u$ two parabolic visual metrics. For vertical geodesics $p,q$ suppose $d_\ell(p,q)=a^{t_\ell}$ and $d_u(p,q)=a^{-t_u}$ then we call any $t_0\in [t_u,t_\ell]$ 
a coarse height at which $p$ and $q$ come close. %\textcolor{red}{$-t_u$ or $t_u$?}
\end{definition}
Since we will only use this notion coarsely and when $\delta(p,q)$ is small it will not matter which coarse height in the interval we chose. In particular it suffices to look at one endpoint or the other.

In our three families of groups, $\delta$ also has nice algebraic interpretations:

\begin{itemize} 
 \item {\textbf{lamplighter groups $L_n$:}} For $p=(x_i)$,\ $q=(y_i) \in \InfSum$, define $d_\ell(p,q)=n^{-l_+}$ where $l_+$ is the maximal integer such that $x_i=y_i$ for all $i<l_+$. In other words, $l_+$ is the smallest integer such that $x_i\neq y_i$. Similarily $d_u(p,q)=n^{l_-}$ where $l_-$ is the minimal integer such that $x_i=y_i$ for all $i> l_-$. Outside this interval of length $l_--l_+$, $x_i= y_i$, i.e. this interval is the support of $p-q$ and so we denote it by $supp(p-q)$ and its length by $|supp(p-q)|$. Then   $\delta(p, q):=n^{|supp(p-q)|}$. 

\item {\textbf{lattices in SOL:}} If $p=(x_1, x_2),\ q=(y_1, y_2) \in B \subset \R^2$ then
we have $\delta(p, q):=|x_1-y_1|\cdot |x_2-y_2|=|\Delta x||\Delta y|$. Because of how $B$ embeds in $\R^2$ we can have $|\Delta x|, |\Delta y|$ arbitrarily small but never zero unless $p=q$.

\item {\textbf{solvable Baumslag-Solitar groups:}} Here we have 
$\delta(p,q):=|p-q|_{\mathbb{R}}|p-q|_{\mathbb{Q}_{n}}.$
Note that if $p,q \in \mathbb{Z}[1/n]$ and if we write $p-q=rn^k$ where $r$ contains no powers of $n$ then $\delta(p,q)= |r|n^k \cdot n^{-k} = |r|$.

\end{itemize}

\begin{definition}
    An \emph{$(\epsilon,M)$-quadrilateral} is a collection of four distinct points $p_i \in B$ for $i=1,2,3,4$ with $\delta(p_i,p_{i+1})\leq \epsilon$ and $\delta(p_i,p_{i+2})\geq M$ where the indices are taken mod $4$. We represent these points in a matrix
    \begin{equation*}
P = 
\begin{bmatrix}
p_{1} & p_{2} \\
p_{4} & p_{3}
\end{bmatrix}.
\end{equation*}
In the case that $p_1+p_3=p_2+p_4$ we call $P$ an $(\epsilon, M)$-parallelogram instead.    
\end{definition}

\begin{figure}\label{fig_quad}
    \begin{center}
           \includegraphics[scale=0.16]{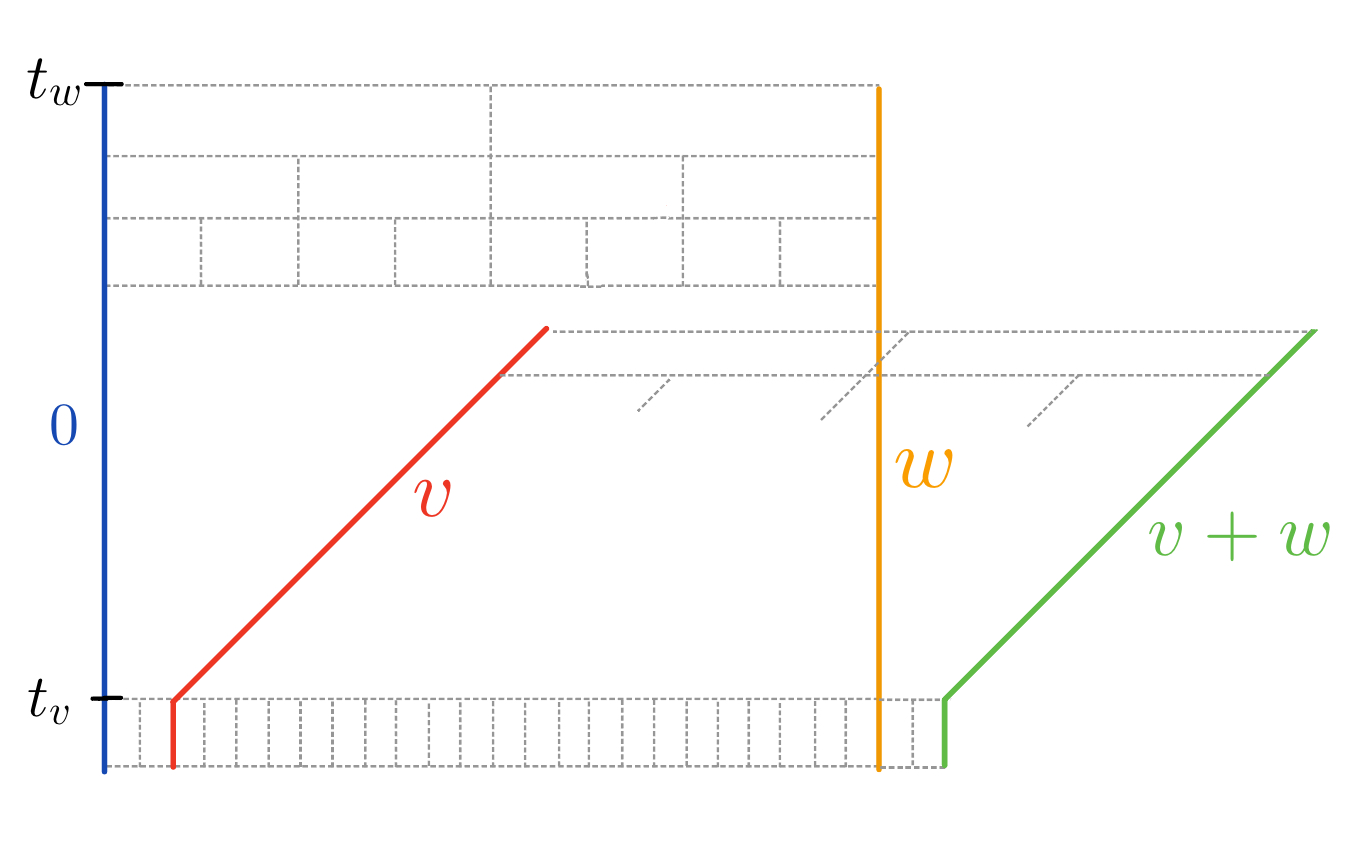}
    \caption{A schematic of an $(\epsilon,M)$ quadrilateral in the model space for $BS(1,2)$, given by vertical geodesics $0,v,w,$ and $v+w$. Every dashed grey segment has length 1. The coarse height $t_v$ is the height where the geodesics 0 and $v$ are closest together. At $t_v$, we can see that the distance between $0$ and $v$ is small and the distance between $w$ and $v+w$ is small but the distance between $0$ and $v+w$ is large. The distance between $0$ and $w$ is realized at height $t_w$, which is also a small distance. 
    }
    \label{fig_quad}
    \end{center}
\end{figure}

We will always choose $M \gg \epsilon$. In this case we have a nice geometric interpretation: an $(\epsilon,M)$-quadrilateral is one for which the vertical geodesics represented by $p_i$ come coarsely close cyclically but are far across the diagonal. The only way this is possible is that they are arranged as in Figure \ref{fig_quad}, 
that is the height at which the pairs $p_1,p_2$ and $p_3,p_4$ come close is very far from the height at which $p_2, p_3$ and $p_4, p_1$ come together.  
Note that rotating the entries of an $(\epsilon, M)$-parallelogram results in another $(\epsilon, M)$ parallelogram.

\begin{lemma}\label{lemma:quadstruct}
    If $X= X_1\times_h X_2$ is a horocyclic product of $CAT(-1)$ spaces and $P$ an $(\epsilon, M)$ quadrilateral for $M \gg \epsilon$ with entries $p_i$ for $i=1,2,3,4$ then if $t_i$ is a coarse height along which $p_i$ and $p_{i+1}$ come close then up to relabeling $t_1\gg t_2$, $t_3 \gg t_4$ while $t_2\ll t_3$, $t_4 \ll t_1 $ and $t_i,t_{i+2}$ are close. In other words the geodesics form a picture very similar to Figure \ref{fig_quad}. %\textcolor{red}{what the last sentence means?}
\end{lemma}
\begin{proof}
    Similar lemmas have appeared in other places such as \cite{Eskin2012} but we give a sketch here because our definition is slightly different. Since we are working coarsely, we write $\gtrsim$ instead of $\geq$ to mean there is some bounded error either additive or multiplicative. 
    First we note that $t_i$ and $t_{i+1}$ must lie at sufficiently different heights if $\delta(p_i, p_{i+2})\geq M$. This is because there is always a path of length at most $|t_i-t_{i+1}| + 4\log\epsilon $ joining $p_i$ and $p_{i+2}$. Now without loss of generality suppose $t_1\gg t_2$. For contradiction suppose $t_2\gg t_3$. Note by the triangle inequality, we have 
    $$d_\ell(p_1,p_4) \geq  d_\ell(p_1, p_2) - d_\ell(p_2, p_3) - d_\ell(p_3, p_4) \gtrsim a^{t_1}-a^{t_2}-a^{t_3} \gtrsim a^{t_1}$$
    while $$ d_u(p_1,p_4) \gtrsim a^{-t_3}-a^{-t_2}-a^{-t_1}\gtrsim a^{-t_3}.$$
    But then $\delta(p_1,p_4)\gtrsim a^{t_1-t_3}\gg 0$ which contradicts that $\delta(p_1, p_4)\leq \epsilon$.
 \end{proof}

Next we verify this geometric picture for $(\epsilon, M)$-parallelograms in our three cases.
\begin{proposition} Let $G= B \rtimes \R$ be either a lamplighter group $L_n$, a solvable Baumslag-Solitar group $BS(1,n)$ or  $\Gamma_A$ a lattice in $SOL$. 
Let \begin{equation*}
P = 
\begin{bmatrix}
0 & w \\
v & w+v
\end{bmatrix}.
\end{equation*}
be an $(\epsilon, M)$ parallelogram in $B$ with $\epsilon\ll M$. Let $t_w$ be the coarse height at which the vertical geodesics defined by $0$ and $w$ come close and  similarly let $t_v$ be the coarse height at which the vertical geodesics defined by $0$ and $v$ come close.  Then for $M \gg \epsilon$ we have that $||t_w|-|t_v||\gg 0$ and hence geometrically our parallelogram has the shape denoted in Figure \ref{fig_quad}.
\end{proposition}
\begin{proof}

As defined above $t_w$ can be any point in the interval defined by $\log{d_\ell(0,w)}$ and $-\log{d_u(0,w)}$ and similarly for $t_v$.
Since we are only interested in the coarse picture the precise values of $t_w$ and $t_v$ are unnecessary. 
Again in the lamplighter group the picture is clearest.

\begin{itemize}
\item {\textbf{lamplighter groups $L_n:$}} Consider $w=(x_i)$,$v=(y_i)$ in $\InfSum$ with $supp(w)=[i_w,j_w]$ and $supp(v)=[i_v, j_v]$. These intervals are the heights at which the two geodesics come close to the geodesic defined by $0$.
The condition that $\delta(0,w), \delta(0,v)$ are small here means that both $|i_w-j_w|$ and $|i_v-j_v|$ are small,  while $\delta(w, v)$ large means that the size of the support of $w-v$ is large. 
More specifically for $\epsilon \ll M$, we can assume that up to exchanging $w$ and $v$ we have that $supp(v-w)=[i_w, j_v]$. This means that $|i_w-j_v|$ is large so that the vertical geodesics defined by $w$ and $v$ come close at very different heights. By translation this means that $w, w+v$ and $v, w+v$ also come together at very different heights and form the geometric picture from \Cref{fig_quad}.

\item {\textbf{solvable Baumslag-Solitar groups: }} Let $w, v\in \Z[\frac{1}{n}]$ where $w=r_wn^{k_w}$ and $v=r_v n^{k_v}$. Then assume that $\delta(0,w)=|r_w|$ and $\delta(0,v)=|r_v|$ are small. By relabeling we may assume $k_v \geq k_w$. Then $w-v= (r_w- r_vn^{k_v-k_w})n^{k_w}$ and $\delta(w,v)=(r_w- r_vn^{k_v-k_w})$. If $\delta(w,v)$ is large then $k_v\gg k_w$ since $r_w$ and $r_v$ are small.  Now $k_w,k_v$ are heights at which the vertical geodesics defined by $w,v$ come close to the geodesic defined by $0$. In fact they are exactly the heights at which these geodesics branch in the tree direction from the vertical geodesic defined by zero. Again we have a quadrilateral like Figure \ref{fig_quad}.

\item {\textbf{lattices in SOL: }} 
If $w=(x_w, y_w), v=(x_v, y_v)\in B \subset \R^2$ then $\delta(0,w)=|x_w||y_w|$ and $\delta(0,v)=|x_v||y_v|$
while $\delta(w, v)= |x_w - x_v||y_w - y_v|$. Without loss of generality we can assume that $|x_w - x_v|$ is large. Then the difference between $t_w=\log{|x_w|}$ and $t_v=\log{|x_v|}$ is large so the two geodesics defined by $w$ and $v$ come close at very different heights.

\end{itemize}

\end{proof}

 \subsection{Quasi-isometries of horocyclic products}

As mentioned in the introduction quasi-isometries of our three families of groups $B \rtimes \Z$ are well understood. Up to composition with an isometry and up to bounded distance they fix level sets of $\Z$.  
This was first proved in the $BS(1,n)$ case by \cite{Farb1998} and in the other two cases by Eskin-Fisher-Whyte \cite{Eskin2012,Eskin2013}. They also show that such maps must also coarsely preserve vertical geodesics and  induce biLipschitz maps of the lower and upper boundaries. 
The biLipschitz constants of these boundary maps only depend on the quasi-isometry constants in the following sense:
There exists $N, K$ such that 
$$N/K d_\ell ( p,q) \leq d_\ell(\psi(p),\psi(q)) \leq N K d_\ell(p,q)$$
and  $$1/(NK) d_u (p,q) \leq d_u(\psi(p),\psi(q)) \leq  K/N d_u(p,q)$$
where $K$ only depends on the quasi-isometry constants. By choosing an appropriate isometry to compose with we can also assume that $N=1$. In addition any two biLipschitz maps of the two boundaries define a quasi-isometry of our group.  This allows us to identify the quasi-isometry groups of our three groups as
\begin{itemize}
\item $QI(\Gamma_A)= (Bilip(\R) \times Bilip(\R)\rtimes) \Z_2 $
\item $QI(BS(1,n))= Bilip(\R) \times Bilip(\Q_n) $
\item $QI(DL(n,n))= (Bilip(\Q_n) \times Bilip(\Q_n))\rtimes \Z_2 $
\end{itemize}
where $Bilip(\R)$ and $Bilip(\Q_n)$ denote the groups of biLipschitz self maps of $\R$ and $\Q_n$ respectively. 

The same is true for many other examples of horocyclic products of hyperbolic spaces with Busemann height functions $X_1 \times_h X_2$. Any time we can show that height level sets are  coarsely preserved we can conclude that  
$$QI(X_1\times_h X_2) \simeq Bilip(\partial_\ell X_1) \times Bilip(\partial_\ell X_2)$$
with an additional $\Z_2$ factor in case $X_1=X_2$. For the latest work on this see \cite{Ferragut}.

The following lemmas are immediate from the definition of $\delta$ and $(\epsilon,M)$-parallelograms.
\begin{lemma}
    If $\Psi$ is a coarsely height preserving quasi-isometry of $X_1 \times_h X_2$ inducing a biLipschitz map $\psi$ on the boundaries, and $p,q$ are vertical geodesics then $$ 1/K^2\delta(p,q) \leq \delta(\psi(p), \psi(q)) \leq K^2 \delta(p,q),$$ where $K$ is the uniform biLipschitz boundary constant and depends only on the quasi-isometry constants of $\Psi$. 
\end{lemma}

\begin{lemma}\label{lemma:quadqi}
If $P$ is an $(\epsilon, M)$-quadrilateral in $X_1\times_hX_2$ and $\Psi$ is a coarsely height preserving quasi-isometry inducing a biLipschitz map $\psi$ on the boundary, then $\psi(P)$ is an $(\epsilon', M')$-quadrilateral where $\epsilon' \leq K^2\epsilon$ and $M'\geq M/K^2$ where $K$ depends only on the quasi-isometry constants of $\Psi$.
 \end{lemma}
\section{Comparative Study of Results in Baumlaug Solitar groups and SOL} \label{Section: BSandSOL}

In this section, we review the rigidity results of pattern-preserving quasi-isometries for lattices in SOL and solvable Baumslag-Solitar groups.
We present the proofs of these rigidity results in \cite{Taback2000, Schwartz1996} from the geometric perspective in terms of the horocyclic product as a comparative study to the lamplighter group.

\subsection{Pattern rigidity for lattices in SOL}
Recall that a matrix $A\in SL(2, \Z)$ which is diagonalizable with no eigenvalues on the unit circle defines a lattice  $\Gamma_{A}=\mathbb{Z}^{2}\rtimes \Z$ in SOL$=\R^2\rtimes \R$. We equip $SOL$ with the standard Riemannian metric $ds^2= e^{-2t}dx^2 + e^{2t}dy^2 +dt^2$.  Recall that the lattice embeds in $SOL$ using the diagonalizing matrix and hence identifies the two eigenspaces of $A$ with the two coordinate spaces $x$ and $y$.
The cosets of the vertical subgroup in $\Gamma_A$ are identified by their $\R^2$ coordinates.

In this section, we recall Schwartz's proof \cite{Schwartz1996} in the horocyclic product model of lattices in SOL and prove that if a quasi-isometry preserves cosets of the vertical subgroup $\langle (0,0,1) \rangle < \Gamma_Z$ then the map is affine. Careful readers may notice that the notation here is different from those in \cite{Schwartz1996}. We focus more on the geometry of the horocyclic product and use $(\epsilon,M)$-quadrilateral to do the argument, which is coarsely the same as the quadrilaterals used in Schwartz's proof.

\begin{proposition}\label{prop:schwartz}\cite{Schwartz1996}
    Let $\Psi$
    be a quasi-isometry that coarsely permutes the left cosets of the vertical subgroup of $\Gamma_A$. Then for any $\epsilon> 0$ there is an $M\gg \epsilon$ that depends on $\epsilon$ and the quasi-isometry constants of $\Psi$ such that 
    if $P$ is an $(\epsilon,M)$-parallelogram the $\Psi(P)$ is also a parallelogram. 
\end{proposition}

\begin{proof}
For simplicity, assume that 
$$\Psi(P)=\begin{bmatrix}
    0 & b \\ c & d
\end{bmatrix}.$$
This is always possible by post composing by an affine map. 
It suffices to prove that $d=b+c$. 

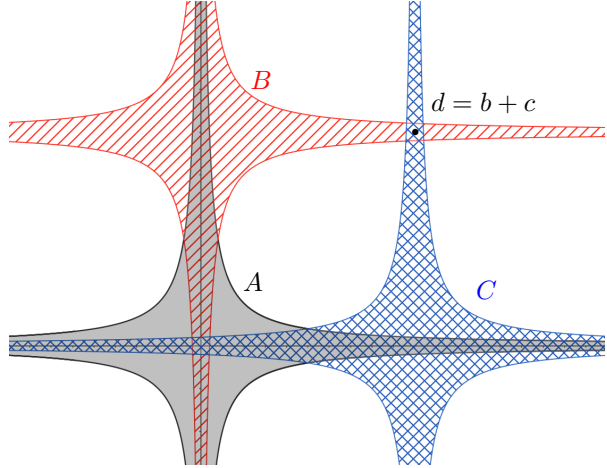
\begin{figure}[htbp]\label{hypfig}
\begin{center}
 \begin{overpic}[abs,unit=1mm,scale=.2]{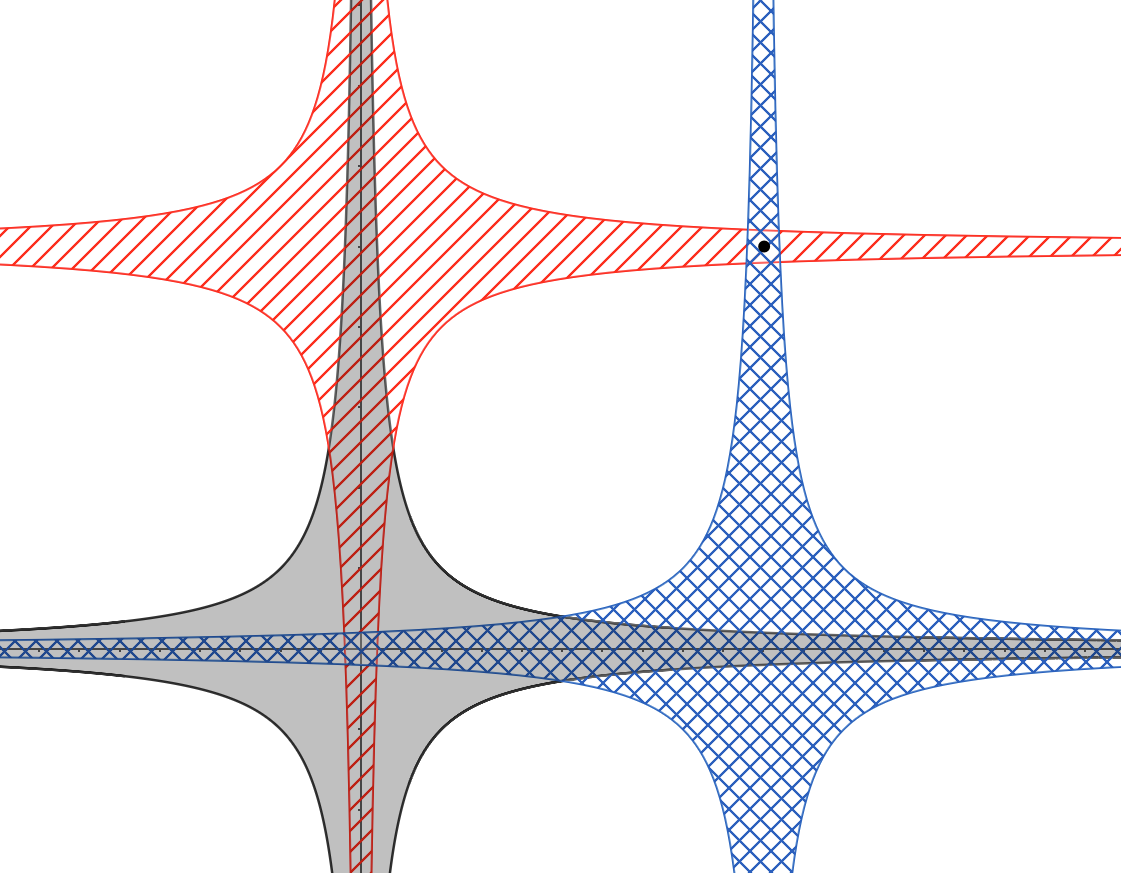}
 \put(31,23){$A$}
 \put(32,50){\color{red}$B$}
 \put(62,22){\color{blue}$C$}
 \put(56,47){$d=b+c$}
     
 \end{overpic}
\caption{The shaded region, $A$, bounded by the black hyperbolas 
 centered around the origin consists of points $p$ with $\delta(0,p)<\epsilon'$. While regions $B$(red and dashed)  and $C$ (blue and cross-hatched)  are those points $p$ with $\delta(b,p)<\epsilon'$ and $\delta(c,p)<\epsilon'$ respectively. If $M$ is large enough then the intersection of $B$ and $C$ has two regions, and each of these regions contains only one lattice point. %One that contains $0$ and one which contains $d=b+c$. 
 }
\label{hypfig}
\end{center}
\end{figure}

Recall by Lemma \ref{lemma:quadqi} that an $(\epsilon,M)$-quadrilateral gets mapped to an $(\epsilon', M')$-quadrilateral where $\epsilon' \leq K^2\epsilon$ and $M'\geq M/K^2$.

Observe that both $p=d$ and $p=b+c$ satisfy the following two inequalities:
$$\delta(p,b)<\epsilon', \quad \delta(p,c)<\epsilon'$$
These inequalities define two regions of the plane $B$,$C$ bounded by hyperbolas (see Figure \ref{hypfig}).
The solution of these two inequalities lies in the intersection of the regions. For fixed $\epsilon'$ we can increase  $M'$ so that if $\delta(b,c)> M'$ then the intersection of two regions consists of two disjoint regions. One region contains $0$ and the other contains both $d$ and $b+c$. By increasing $M$ and hence $M'$ we can shrink these regions of intersection and ensure that both regions only contain a single lattice point, in other words that $d=b+c$.
 Hence, for every $\epsilon$ there exists a sufficiently large $M$ such that when $P$ is an $(\epsilon,M)$-parallelogram then $\Psi(P)$
 is an $(\epsilon',M')$-quadrilateral and hence also a parallelogram.
\end{proof}

In order to show that all parallelograms are mapped to parallelograms, Schwartz shows that
you can ``generate" all parallelograms with `large' ones. In our notation, this means that for any $(\epsilon,M)$, we can find a generating  set $\Sigma  \subset B$ so that for any $v,w \in \Sigma $ we have that
$$\begin{bmatrix}
     0  &v\\
w  &  v+w
\end{bmatrix}$$
is an $(\epsilon, M)$-parallelogram.

Then for any parallelogram
$$P=\begin{bmatrix}
     a & b \\ c & d 
\end{bmatrix}$$
first write  $c-a=d-b=v_1+ \cdots v_k$ where $v_i \in \Sigma $ and define 
$$P_j= \begin{bmatrix}
     a + v_1+  \cdots +v_{j-1}   & b + v_1+  \cdots +v_{j-1}\\
a+ v_1+  \cdots+ v_{j-1}+v_{j} & b + v_1+  \cdots +v_{j-1} + v_j
\end{bmatrix}.$$
If $\Psi(P_j)$ is a parallelogram for all $j=1,\ldots k$ 
then $$\psi(a) +  \psi(b + v_1) =\psi(b) + \psi(a+ v_1)$$
$$\psi(a+v_1) +  \psi(b + v_1+ v_2) =\psi(b+v_1) + \psi(a+ v_1+v_2)$$
$$\vdots$$
$$\psi(a+v_1 + \cdots v_{k-1}) +  \psi(b + d- b) =\psi(b+v_1+ \cdots v_{k-1}) + \psi(a+ c -a).$$
By adding all equations we get $\psi(a) +  \psi(d) =\psi(b) + \psi(c)$ and so $\Psi(P)$ is also a parallelogram. 
To show that $\Psi(P_j)$ is a parallelogram, rotate the entries by $\pi/4$ and repeat the above argument for each $P_j$.  
For example $P_1$ becomes 
$$P'_1= \begin{bmatrix}
     a + v_1  & a \\
     b + v_1 & b \\
\end{bmatrix}=
\begin{bmatrix}
     b + (a - b) + v_1  & b + (a-b) \\
     b + v_1 & b \\
\end{bmatrix}
$$
and if we write $a+v_1- (b+v_1)=a-b=v_{11} + \cdots + v_{1k_1}$ as a sum of generators in $\Sigma$ then 
$$P_{11}'= \begin{bmatrix}
      b+v_1   & b\\
 b+v_1+ v_{11}&  b+v_{11}
\end{bmatrix},
P_{12}'=
\begin{bmatrix}
      b+v_1+ v_{11}  & b+v_{11}\\
 b+v_1+ v_{11} + v_{12}  &  b+v_{11}+v_{12} 
\end{bmatrix}, \cdots$$
all of which are $(\epsilon, M)$-parallelograms and so the result follows.

\subsection{Rigidity of solvable Baumslag-Solitar groups}
Next we rewrite and simplify Taback's proof from \cite{Taback2000} using $(\epsilon,M)$-quadrilaterals.

\begin{proposition}\label{prop:taback}\cite{Taback2000}
Let $\Psi$
    be a quasi-isometry that induces a biLipschitz map $\psi$ that coarsely permutes the cosets of the vertical subgroup of $BS(1,n)$. Then for any $\epsilon> 0$ there is an $M\gg \epsilon$ that depends on $\epsilon$ and the biLipschitz constants of $\psi$ such that 
    if $P$ is an $(\epsilon,M)$-parallelogram in $B=\Z[\frac{1}{n}]$ then $\psi(P)$ is also a parallelogram. 
    \end{proposition}

\begin{proof}

Suppose that $P$ is an $(\epsilon,M)$-parallelogram in $B=\Z[\frac{1}{n}]$. Then by Lemma \ref{lemma:quadqi} we know that $\psi(P)$ is
an $(\epsilon',M')$-quadrilateral where $\epsilon'\leq K^2\epsilon$ and $M'\geq M/K^2$. 
Let $$\psi(P)=\begin{bmatrix}
     p_1 & p_2 \\ p_4 & p_3 
\end{bmatrix}$$
and for $i=1,2,3,4$ write $p_{i+1} -p_i = r_in^{k_i}$ where $r_i$ is an integer not divisible by $n$.
Then we know that the $\delta(p_i,p_{i+1})=|r_i| \leq \epsilon'$  but $|k_i - k_{i+1}|$ is large for all $i$ (mod $4$) by Lemma \ref{lemma:quadstruct}. Without loss of generality we may assume that $k_1>k_2$ and $k_3>k_4$.
Using this information we note that
\begin{eqnarray*}
    p_{3}-p_1&=&(p_{3}-p_{2}) + (p_{2}-p_1)\\
    &=&r_{2}n^{k_{2}}+ r_1n^{k_1}=n^{k_{2}}\left( r_2 + r_1n^{k_1-k_{2}}\right)
\end{eqnarray*}
while 
\begin{eqnarray*}
    p_{1}-p_3&=&(p_{1}-p_{4}) + (p_{4}-p_3)\\
    &=&r_{4}n^{k_{4}}+ r_3n^{k_3}=n^{k_{4}}\left( r_4 + r_3n^{k_3-k_{4}}\right)
\end{eqnarray*}
Now since $k_1 > k_2$ and $k_3>k_4$ we know that $r_2 + r_1n^{k_1-k_{2}}$ and $r_4 + r_3n^{k_3-k_{4}}$  are integers not divisible by $n$ and so comparing the above two expressions we must have $k_2=k_4$ and
$r_2 + r_1n^{k_1-k_{2}}=-r_4 - r_3n^{k_3-k_{4}}$. By symmetry we also have that $k_1=k_3$. 
This allows us to conclude that $r_2=-r_4$ and $r_1=-r_3$.
and so  $$p_2-p_1= r_1n^{k_1}= -r_3n^{k_3}=p_3-p_4$$ so that $p_1+p_3=p_2+p_4$ as desired. 

\end{proof}
To show that $\psi$ sends all parallelograms to parallelogram, a similar argument to the one for SOL is used to show that any parallelogram can be written as a sum of `large' parallelograms.
This is possible essentially because we can write $n^{k+1}$ as a sum of $n$ copies $n^k$ so that a vertical geodesic that branches from the $0$ geodesic at height $k+1$ can be written as a sum of vertical geodesics that branch at height $k$. In the lamplighter case this is not possible and so it is not always possible to generate all parallelograms with $(\epsilon, M)$-parallelograms. For more details see Section \ref{Section: MainResults}.
\section{Rigidity and Non-rigidity in Lamplighter groups} \label{Section: MainResults}

In this section, we discuss pattern preserving  quasi-isometries in Lamplighter groups. In \Cref{subsec:lamplighter_patterns}, we first show that, up to swapping tree factors, a pattern preserving isometry of $DL(2,2)$ is induced by an generalized affine map of $B=\InfSumtwo$. The analogous fact about pattern preserving isometries of the standard model spaces for $BS(1,n)$ and lattices in SOL is immediate as these model spaces are more rigid then $DL(2,2)$.  In Section \ref{subsec:large_grams} we show that pattern preserving quasi-isometries of $DL(2,2)$ are ``affine on the large scale'' similar to Propositions \ref{prop:schwartz} and \ref{prop:taback} for lattices in SOL and $BS(1,n)$.  However, unlike in $BS(1,n)$ and lattices in SOL, this large scale behavior cannot be extended to all parallelograms. 
In \Cref{subsec:counterexample}, we construct a family of pattern preserving quasi-isometries and give a counterexample to show that the pattern preserving condition no longer implies the same rigidity that we saw in the SOL and Baumslag-Solitar cases. 

\subsection{Pattern Preserving Isometries in Lamplighter Groups} \label{subsec:lamplighter_patterns}

In this section, we study pattern preserving isometries of the Diestel-Leader graph, focusing on the case $DL(2,2)$ which is the model space for the lamplighter group $L_2$. All isometries of $DL(2,2)$ can be written as a composition of the following three types isometries: an isometry that swaps the tree factors (we call any such isometry an ``orientation reversing isometry"), an isometry that translates up (or down) in the first tree and down (or up) by the same amount in the second tree and an isometry that fixes height and orientation.  Specifically we have 
$$Isom(DL(2,2))= ((Isom (\Q_2) \times Isom(\Q_2)) \rtimes \Z) \rtimes \Z_2.$$
Thinking of $DL(2,2)$ as a Cayley graph, the first type of isometry cannot be induced by left multiplication by a group element while the second type can be chosen to be induced by left multiplication by elements of the vertical subgroup $\left< t \right>$. The last type contains both isometries that are induced by left multiplication by $B= \InfSum$ and also many others that are not induced by left multiplication.

We can now prove Theorem \ref{thrm:affineisometry}. Namely that if $\Psi: DL(2,2)\to DL(2,2)$ is an orientation preserving pattern preserving isometry then there is some $g \in L_2$ such that $\Psi$ is bounded distance from left multiplication by $g$. We will give a geometric proof of this proposition here but the result also follows from Proposition \ref{prop:algisom} whose proof relies on understanding $\delta$.

\begin{proof}\emph{(Of Theorem \ref{thrm:affineisometry})}
Recall that orientation preserving means that the two tree factors are not interchanged. This assumption is necessary since no group element can interchange these factors. Any orientation preserving isometry that preserves the set of cosets can be composed with the action of a group element to obtain an isometry that sends the identity $\langle t \rangle$ coset to itself. Similarly we can compose such an isometry with a group element acting by translating up or down to obtain an isometry that  preserves height. Thus it is sufficient to  assume that $\Psi$ is an isometry that preserves orientation and height, and that the induced map $\psi$ permutes the cosets,  and fixes the identity coset and to show that $\Psi$ must be the identity map.

\begin{figure}[h]
    \centering
    \includegraphics[scale=.23]{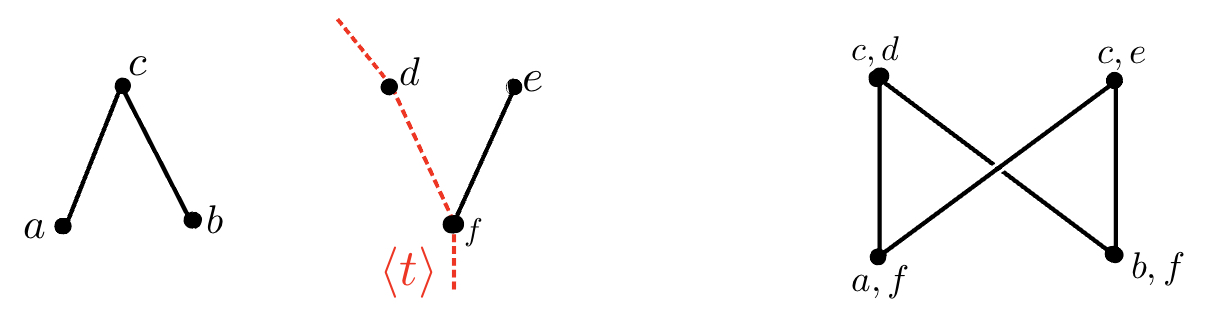}
    \caption{A height one portion of $DL(2,2)$ on the right and its tree factors on the left. Vertex labels illustrate the proof of Theorem \Cref{thrm:affineisometry}. The red (dotted) line represents the vertical geodesic corresponding to the identity coset in the right-hand factor.  }
    \label{fig_prooffig}
\end{figure}

Now, suppose that $\Psi$ is not the identity, and consider the action of $\Psi$ on the tree to the left, in the horocylic product. There exists some point $a$ which is not fixed, but is the child of some fixed point $c$. Because $\Psi$ acts as a height preserving isometry, this means that $a$ and $b$ (the two children of $c$) are permuted (See \Cref{fig_prooffig}).

We are also assuming that the identity coset $\langle t\rangle$ is preserved by $\psi$. The identity coset is represented by a vertical geodesic in $DL(2,2)$. This geodesic appears as a pair of vertical geodesics one in each tree. In the right-hand tree, there is some segment of this geodesic at the same height as the edge from $c$ to $a$. Label the vertices of this segment $d$ and $f$ and the other child of $f$ as $e$ (See \cref{fig_prooffig}). Note that $f$ and $d$ are fixed point-wise by $\Psi$ because the identity coset is fixed by $\psi$ and $\Psi$ is height preserving. Hence, $e$ is also fixed because $\Psi$ is an isometry. 

We have now produced a quadrilateral $(c,d), (b,f), (c,e), (a,f)$ in $DL(2,2)$ of height one, where $\Psi$ permutes $(a,f)$ and $(b,f)$, but leaves $(c,d)$  and $(c,e)$ fixed. However, because cosets of $\langle t \rangle$  are vertical geodesics that partition the vertices of $DL(2,2)$, there are exactly two cosets, passing through any height one quadrilateral. Any pattern preserving isometry must either fix both cosets or permute them. If both cosets are fixed then all four vertices are fixed, but in our constructed quadrilateral $(a,f)$, and $(b,f)$ are not fixed so this is not the case for our quadrilateral. If the cosets are permuted, then opposite edges of the quadrilateral are permuted. In particular this means that the edges $(c,d)-(a,f)$ and $(c,e)-(b,f)$ are permuted and the vertices that $(c,d)$ and $(c,e)$ are permuted. This again contradicts our construction of this quadrilateral. Thus, under our assumptions about $\Psi$, such a quadrilateral is impossible, providing a contradiction.

\end{proof}

This proposition shows that our pattern preserving assumption is a strong hypothesis in the case where the quasi-isometry is an isometry. We now consider quasi-isometries that are not isometries, and try to imitate the arguments made in the case of lattices in $SOL$ and $BS(1,n)$. We will see similar behavior carry over for $(\epsilon, M)$ parallelograms for $M\gg \epsilon$, but in general quasi-isometries of $DL(2,2)$ are not affine, in contrast to the behavior we saw for the other two groups. 

\subsection{Large Parallelograms}\label{subsec:large_grams}
Recall that in the lamplighter case $log_2(\delta(p,q))$ has a very nice interpretation. Namely
if $p=(x_i), q=(y_i) \in \InfSumtwo$ and 
$x_i=y_i$ for all $i> \ell_-$ and all $i< \ell_+$. Then $log_2(\delta(p,q))=\ell_{-}-\ell_{+}$, that is, it measures the length of the smallest interval such that $x_i=y_i$ outside this interval. This interval is denoted $supp(p-q)$.

\begin{proposition}\label{Prop:eMlamplighter}
Let $\Psi$
    be a quasi-isometry and $\psi$ the induced biLipschitz map that coarsely permutes the cosets of the vertical subgroup of $L_2$. Then for any $\epsilon> 0$ there is an $M\gg \epsilon$ that depends on $\epsilon$ and the quasi-isometry constants of $\Psi$ such that 
    if $P$ is an $(\epsilon,M)$-parallelogram in $B=\InfSumtwo$ then $\psi(P)$ is also a parallelogram. 
    \end{proposition}

\begin{proof}

    By Lemma \ref{lemma:quadqi} we know that if $P$ is an $(\epsilon, M)$-quadrilateral then $\psi(P)$ is
an $(\epsilon',M')$-quadrilateral where $\epsilon'\leq K^2\epsilon$ and $M'\geq M/K^2$.
We need to show that for any $\epsilon$ there is an $M$ such that if $P$ is an $(\epsilon',M')$ quadrilateral
then $P$ is a parallelogram.  In fact it is enough to choose the constants so that $\log{M'}> 2\log{\epsilon'}$.
\begin{claim}
    Let  $P$ be an $(2^S,2^{2S})$ quadrilateral then $P$ is a parallelogram.
\end{claim}
    This claim proves Lemma \ref{Theorem: Main} from the introduction and
    follows nicely from the interpretation of $\log_2(\delta(p,q))=supp(p-q)$. Suppose 
    $$P=\begin{bmatrix}
 a&b\\
 c&d
 \end{bmatrix}.$$
Then 
 \[ max(supp(a-b),supp(a-c)) \leq S\]
 \[ min(supp(a-d), supp(b-c)) \geq 2S\] 

Next we will show that  
$supp(a-b)\cap supp(a-c)=\emptyset$. %and likewise $Supp(a,b)\cap Supp(b,d)=\emptyset.$$
Consider the smallest integer $N_+$ such that $n$ is outside $supp(a-b)\cup supp(a-c)$ for all $n\geq N_+$. We must have $a_n=b_n=c_n$ and so $n$ must not be in $supp(b-c)$. Similarly, if $N_-$ is the largest integer such that for all $n\leq N_{-}$, $n$ is outside  $supp(a-b)\cup supp(a-c)$, then $n$ cannot be in $supp(b-c)$.  This means the largest integer in $supp(b-c)$  and the smallest integer in $supp(b-c)$ are both in $supp(a-b)\cup supp(a-c)$. See \Cref{fig_DisjointSupport}. Hence, $supp(a-b)\cap supp(a-c)=\emptyset$, as if  these subsets of integers overlapped then $supp(b-c)$ would be too small, which contradicts the assumption that the support is at least $2S$.

Note that by symmetry we also have that $supp(b-a)\cap supp(b-d)=\emptyset, supp(d-b)\cap supp(d-c)=\emptyset$ 
and $supp(c-a)\cap supp(c-d)=\emptyset$.

Now, fix a given $i$. We would like to show that for every $i$, the equation $a_i+d_i=c_i+b_i$ holds. If $a_i=b_i=c_i=d_i$ then we are done, so assume without loss of generality that $a_i\neq b_i$. That means $i\in supp(a-b)$, and by the above claim this implies that $i\notin supp(a-c)$ and $i\notin supp(b-d)$ hence that $a_i=c_i$ and $b_i=d_i$. This shows that $a_i+d_i=c_i+b_i$ as desired. 
\end{proof}

\begin{figure}
    \centering
    \includegraphics[scale=0.21]{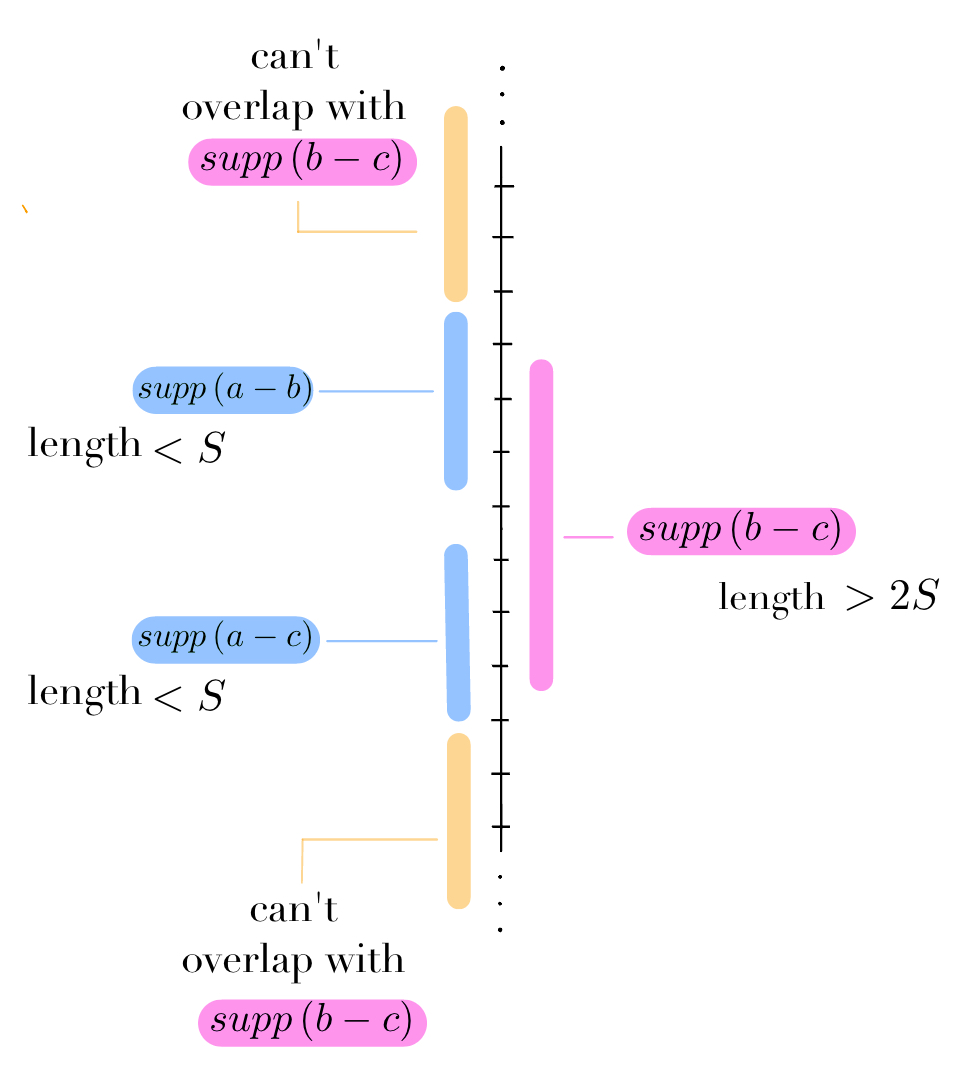}
    \caption{
    In an $(2^S,2^{2S})$-parallelogram formed by $a,b,c,d$, we know that  $supp(a-b)$ and $supp(a-c)$ both have length at most $S$, while $supp(b-c)$ has length at least $2S$. This implies that $supp(a-c)\cap supp(a-b)=\emptyset$.} 
    \label{fig_DisjointSupport}
\end{figure}

We have now shown that the map $\psi$ is parallelogram preserving when restricted to $(\epsilon,M)$-parallelograms for $M \gg \epsilon$.
This mirrors the SOL and Baumslag Solitar cases. 
However, in the SOL and Baumslag-Solitar case, the result generalizes to all parallelograms and not only ones that are `large' enough. This is done by finding a generating set $\Sigma$ for $B$ such that all parallelograms based on elements in $\Sigma$ are $(\epsilon,M)$ parallelograms with $M\gg\epsilon$.

In the lamplighter case, this argument only works when $(\epsilon,M)$ is very small. Namely a generating set for $B$ must contain a generator for $\Z_2$ at each index $i$. Therefore the best one can hope for is that quadrilaterals based on $\Sigma$ are $(\epsilon, M)$ quadrilaterals for $\epsilon=2^1$ and $M=2^2$. 
In addition, such quadrilaterals are useful only in the case that $\Psi$ is a $(1,0)$-quasi-isometry, i.e. an isometry. This gives us a new proof of Theorem \ref{thrm:affineisometry}.

\begin{proposition}\label{prop:algisom} If $\Psi:L_2\to L_2$ is a $(1,0)$ pattern preserving quasi-isometry, then $\psi$ must be parallelogram preserving.

\end{proposition}

\begin{proof}
Consider the set $\Sigma$ of elements of $\InfSumtwo$  where the size of the support is equal to 1. Given two elements $w,v$ in $\Sigma$, we can consider the parallelogram 
$$\begin{bmatrix}
 a&a+w\\
 a+v&a+w+v
 \end{bmatrix}.$$
 If $w=v$ then (because $n=2$), we get $\psi(w+v+a)=\psi(2w+a)=\psi(a)$ so $\psi(a+w+v)+\psi(a)=2\psi(a)=0=\psi(a+w)+\psi(a+v)$.
  If $w\neq v$ then this is a $(\epsilon,M)$-parallelogram for $\epsilon=2^1$ and $M=2^2$ and $\psi$ is parallelogram preserving with respect to this parallelogram. 
  Every element of $\InfSumtwo$ can be written as a sum of elements of $S$, so the parallelograms of the above form generate all possible parallelograms. 
\end{proof}

Notice that as quasi-isometry constants get worse we need $M \gg \epsilon$ in order for the image quadrilateral to be a $(2^S, 2^{2S})$ quadrilateral. 
In SOL and $BS(1,n)$ this does not matter as much, since we can find a generating set $\Sigma$ such that parallelograms based on this set are all $(\epsilon,M)$-parallelograms.  However, in the lamplighter case 
the situation is different.  Pick $\epsilon, M$ such that $\log_2 M > 2 \log_2 \epsilon + 1$. 
Suppose for contradiction that there is a generating set $\Sigma$ for $\InfSumtwo$ such that any two elements $v,w \in \Sigma$ generate an $(\epsilon, M)$-parallelogram
$$P=\begin{bmatrix}
 0& w\\
 v&w+v
 \end{bmatrix}.$$
 
 Since $supp(v+w)$ is an interval that is larger than the sum of the sizes of $supp(v)$ and $supp(w)$ there must be some index $i_0$ not in $supp(v)\cup supp(w)$ so there must be another generator $z \in \Sigma$ with $i_0 \in supp(z)$ but then without loss of generality $z$ and $v$ will not generate an $(\epsilon, M)$-parallelogram.

In fact, in the next section we show that we can generate counterexamples indicating that for larger quasi-isometry constants, the map $\psi$ may not be parallelogram preserving for all parallelograms. 

\subsection{A Family of Quasi-isometries and a Counterexample}\label{subsec:counterexample}
In this section we prove Theorem \ref{thrm:counterexamples}. 
Here is a way to create many examples of maps $\psi:\InfSumtwo\to \InfSumtwo$ that are bijective and biLipschitz on both boundaries. Such a map will induce a quasi-isometry $\Psi: DL(2,2)\to DL(2,2)$.

Suppose that $S_m$ is the set of sequences of 0's and 1's of length $m$, and let $\pi$ be an element of the symmetric group on $2^m$ elements that permutes the elements of $S_m$.  This induces a map $\psi$, $\psi:\InfSumtwo\to \InfSumtwo$ by applying $\pi$ to the subsequence with indicies between 0 and $m-1$ and the identity map at all other heights. 

\begin{claim} Any such map $\psi:\InfSumtwo\to \InfSumtwo$  induced by $\pi\in S_{2^m}$ will be biLipschitz in both the upper and lower boundaries, and the biLipschitz constants of the boundary maps  will both be less than or equal to $2^m$.
\end{claim}

We will show this map is biLipschitz in the upper boundary with metric $d_u$. The argument for the lower boundary component is similar. Consider two arbitrary elements $p$ and $q$ in $\InfSumtwo$. Let $\ell_+$ be the maximal integer where $x_i=y_i$ for all $i<\ell_+$. If $\ell_+$ is either less than $0$ or greater than $m-1$, then $d_u(p,q)=d_u(\psi(p),\psi(q))$ and the claim follows. If $0\leq\ell_+<m$, then $p$ and $q$ do not match at heights between $0$ and $m$ and the permutation $\pi$ applied to the subsequence of $p$ and the subsequence of $q$ has different results. Thus there is some $0\leq j<m$ such that $\psi(x)_j\neq \psi(y)_j$. This means that the distance in $d_u$ between $p$ and $q$ has changed by at most a factor of $2^m$ under $\psi$. This implies that $\psi$ is biLipschitz in the upper boundary with biLipschitz constant less than or equal to $2^ m$.

These examples give us many possible maps on $\InfSumtwo$, each of which must induce a quasi-isometry on the Lamplighter group. 
Among these examples are examples where the map is not parallelogram preserving. 

For example consider the map  $\pi$ where $m=3$, which transposes  $100$ to $111$ and acts as the identity otherwise. In this map we see that \[\pi(100)+\pi(001)=111+001=110\]
while
\[\pi(100+001)=\pi(101)=101.\]
So this $\pi$ is not affine and hence induces a map $\psi$ on $\InfSumtwo$ which is not affine.

This family of examples of biLipschitz maps leads us to the following question.

\begin{question}
Can all bijective biLipschitz maps % on $\Infum$ 
be written as compositions of maps generated by permutations in $S_m$ as in the examples above? \end{question}

\bibliographystyle{alpha}
\bibliography{references.bib}

\end{document}